\documentclass[11pt, reqno]{amsart}
\usepackage{amssymb}
\usepackage{latexsym}
\usepackage{amsmath}
\usepackage{amsfonts}
\usepackage[hidelinks]{hyperref}
\hypersetup{
  colorlinks   = true, 
  urlcolor     = blue, 
  linkcolor    = blue, 
  citecolor   = blue 
}
\usepackage{enumitem}
\usepackage{fancyhdr}
\usepackage{cleveref}
\usepackage{mathrsfs}

\usepackage{scalerel}
\usepackage{stackengine}
\stackMath
\def\hatgap{2pt}
\def\subdown{-2pt}
\newcommand\reallywidehat[2][]{%
\renewcommand\stackalignment{l}%
\stackon[\hatgap]{#2}{%
\stretchto{%
    \scalerel*[\widthof{$#2$}]{\kern-.6pt\bigwedge\kern-.6pt}%
    {\rule[-\textheight/2]{1ex}{\textheight}}
}{0.5ex}
_{\smash{\belowbaseline[\subdown]{\scriptstyle#1}}}%
}}

\usepackage[dvipsnames]{xcolor}
\theoremstyle{plain}
\newtheorem{theorem}{Theorem}[section]

\newcommand{\R}{\ensuremath{\mathbb{R_+}}}

\newtheorem{corollary}{Corollary}[theorem]
\newtheorem{lemma}[theorem]{Lemma}

\newtheoremstyle{remark}
    {} 
    {} 
    {}          
    {}          
    {\bfseries} 
    {.}         
    {.5em}      
    {}          

\theoremstyle{remark}

\newtheoremstyle{example}
    {\dimexpr\topsep/2\relax} 
    {\dimexpr\topsep/2\relax} 
    {}          
    {}          
    {\bfseries} 
    {.}         
    {.5em}      
    {}          

\theoremstyle{example}
\newtheorem{example}{\emph{\textbf{Example}}}[section]

\newtheoremstyle{definition}
    {\dimexpr\topsep/2\relax} 
    {\dimexpr\topsep/2\relax} 
    {}          
    {}          
    {\bfseries} 
    {.}         
    {.5em}      
    {}          

\theoremstyle{definition}
\newtheorem{definition}{Definition}[section]

\newtheorem*{similartheorem*}{Theorem \dualnumber{$'$}}

\numberwithin{equation}{section}

\def\R{\mathbb{R}}

\allowdisplaybreaks

\begin{document}
\title[The Fourier transform on Rearrangement-Invariant Spaces]{The Fourier transform on Rearrangement-Invariant Spaces}

{\let\thefootnote\relax}

\thanks{The third author is supported by National Board for Higher Mathematics, Government of India.}

\author{Ron Kerman, Rama Rawat and Rajesh K. Singh}

\address{Ron Kerman: Department of Mathematics, Brock University, St. Catharines, Ontario, L2S 3A1, Canada}
\email{rkerman@brocku.ca}
\address{Rama Rawat: Department of Mathematics and Statistics, Indian Institute of Technology, Kanpur-208016}
\email{rrawat@iitk.ac.in}

\address{Rajesh K. Singh : Department of Mathematics, Indian Institute of Science, Bangalore-560012, Karnataka, India.}
\email{rajeshsingh@iisc.ac.in}

\pagestyle{headings}

\begin{abstract}
We study inequalities of the form
\begin{equation*}
    \rho ( \lvert \hat{f}   \rvert) \leq C \sigma(f) < \infty, 
\end{equation*}
with $f \in L_{1}(\mathbb{R}^n)$, the Lebesgue-integrable functions on $\mathbb{R}^n$ and
\begin{equation*}
    \hat{f}(\xi) := \int_{\mathbb{R}^n}  f(x) \, 
e^{- 2 \pi i \xi \cdot x} dx, \ \ \ \xi \in \mathbb{R}^n.
\end{equation*}
The functionals $\rho$ and $\sigma$ are so-called rearrangement-invariant (r.i.) norms on   $M_{+}(\mathbb{R}^n)$, the nonnegative measurable functions on $\mathbb{R}^n$.

Results first proved in the general context of r.i. spaces are then both specialized and expanded on in the special cases of Orlicz spaces and of Lorentz Gamma spaces.
\end{abstract}

\maketitle

\section{Introduction}\label{Introduction}
\medskip
The Fourier transform $\mathscr{F}(f)$, or $\hat{f}$, of a function $f$ in the class $L_{1}(\mathbb{R}^n)$ of functions integrable on $\mathbb{R}^n$ is given by
\begin{equation}
    \hat{f}(\xi) := \int_{\mathbb{R}^n}  f(x) \, 
e^{- 2 \pi i \xi \cdot x} dx, \ \ \ \xi \in \mathbb{R}^n.
\end{equation}
For a succinct account of its basic properties, see \cite[pp. 91-94]{P02}.

We are interested in inequalities of the form
\begin{equation}\label{Fourier boundedness}
   \rho ( \lvert \hat{f}   \rvert) \leq C \sigma(f) < \infty, \ \ f \in L_{1}(\mathbb{R}^n),
\end{equation}
in which $\rho$ and $\sigma$ are so-called rearrangement-invariant (r.i.) norms on   $M_{+}(\mathbb{R}^n)$, the nonnegative functions in the class $M(\mathbb{R}^n)$ of measurable functions on $\mathbb{R}^n$. The salient feature of an r.i. norm $\rho$ is that

\begin{equation*}
    \rho(f) =\rho(g), \ \ \ f,g \in M(\mathbb{R}^n),
\end{equation*}
 whenever $f$ and $g$ are equimeasurable, namely, 
 \begin{equation*}
 \left| \{x \in \mathbb{R}^n :|f(x)| > \lambda \} \right| = \left| \{x \in \mathbb{R}^n :|g(x)| > \lambda \} \right|, \ \ \  \lambda \in \mathbb{R_+} : =(0,\infty).
\end{equation*}

For  our purposes the fundamental inequality involving $\hat{f}$, proved in \cite[Theorem 4.6]{JT71}, is 
\begin{equation}\label{JT U dominate F}
\int_0^t {(\hat{f})}^{*}(s)^2ds \leq C \int_0^t (U{f}^{*})(s)^2ds,
\end{equation}
with $(U{f}^{*})(s) :=\int_0^{s^{-1}}{f}^{*},$  $ f \in L_1(\mathbb{R}^n)$, $s \in \mathbb{R}_+$. 

In (\ref{JT U dominate F}), ${g}^{*}$ is the nonincreasing rearrangement of $g$ on $\mathbb{R_+}$ defined by 
\begin{equation*}
    {g}^{*}(s)= \mu_g^{-1}(s), \ \ \  s \in \mathbb{R_+},
\end{equation*}
where
\begin{equation*}
\mu_g(\lambda)= \left| \{x  \in \mathbb{R}^n  :|g(x)| > \lambda \} \right|, \ \ \  \lambda \in \mathbb{R_+}.
\end{equation*}

 We observe that $f,g \in M_{+}(\mathbb{R})$ are equimeasurable if and only if $f^*=g^*$ on $\mathbb{R_+}.$

 According to a result of Luxemberg, see \cite[Theorem 4.10, p. 62]{BS88}, an r.i. norm
 $\rho$ on $M_{+}(\mathbb{R}^n)$ can be represented in terms of an r.i. norm $\bar{\rho}$ on $M_{+}(\mathbb{R_+})$ by
\begin{equation*}
    \rho(f) =\bar{\rho}(f^\ast), \ \ \ f \in M_{+}(\mathbb{R}^n).
\end{equation*}

As we'll see, the inequality (\ref{JT U dominate F}) and (\ref{P boundedness on decreasing}) in  \Cref{FT implies U} below ensure that (\ref{Fourier boundedness}) is equivalent to
\begin{equation}\label{reduction to U}
   \bar{\rho}(Uh) \leq C\bar{\sigma}(h),  \ \ \ h \in M_{+}(\mathbb{R}_+),
\end{equation}
 when the r.i. space 
 \begin{equation*}
L_{\bar{\rho}}(\mathbb{R}_+):= \left\lbrace h \in M_+(\mathbb{R}_+) : \bar{\rho}(h)< \infty \right\rbrace,
\end{equation*}
 is an interpolation space between the space $L_2(\mathbb{R}_+)$ of functions square-integrable on $\mathbb{R}_+$ and the space  $L_\infty(\mathbb{R}_+)$ of functions essentially bounded on $\mathbb{R}_+.$ See \Cref{RI.Spaces} below for the definition of an interpolation space.

With the above considerations in mind, we can now describe our approach to the study of (\ref{reduction to U}). Given a specific class of r.i. spaces, we start with an r.i. norm $\bar{\rho}$ in the given class satisfying the relevant interpolation property. Then we form the smallest possible r.i. norm 
${\bar{\sigma}}_{\bar{\rho}}$ satisfying (\ref{reduction to U}) together with $\bar{\rho},$
namely, 
\begin{equation*}
{\bar{\sigma}}_{\bar{\rho}} (h)= {\bar{\rho}}(Uh^*), \ \ \ h\in M_{+}(\mathbb{R}_+).
\end{equation*}

We then determine conditions on an r.i. norm $\bar{\sigma}$, in the same class as $\bar{\rho}$, such that
\begin{equation*}
{\bar{\sigma}}_{\bar{\rho}} (h) \leq  \bar{\sigma} (h), \ \ \ h \in M_{+}(\mathbb{R}_+).
\end{equation*}

The following result shows (1.2) implies (1.4) with only a mild restriction on $\bar{\rho}.$

\begin{theorem}\label{FT implies U}
Let $\rho$ and $\sigma$ be r.i. norms on $M_{+}(\mathbb{R}^n)$ defined in terms of r.i. norms $\bar{\rho}$ and $\bar{\sigma}$ on $M_{+}(\mathbb{R_+})$  by 
\begin{equation}
    \rho(f) = \bar{\rho}(f^{*}) \ \ \ \text{and}   \ \ \ 
    \sigma(f) = \bar{\sigma}(f^{*}), \ \ \ f \in M_+(\mathbb{R}^n).
\end{equation}

Assume that 
\begin{equation}\label{P boundedness on decreasing}
    \bar{\rho}(g^{**}) \leq C  
     \bar{\rho}(g^{*}), 
\end{equation}
in which 
\begin{equation}
g^{**}(t)={t}^{-1} \int_0^t g^{*} , \ \ \ g \in M_{+}(\mathbb{R}_+), \ \ \ t \in \mathbb{R_+}.
\end{equation}
Then,
\begin{equation}
    \rho( \hat{f}) \leq C \sigma(f) , 
\end{equation}
 $f \in ( L_{1}\cap L_\sigma )(\mathbb{R}^n)$, $f(x)=g(|x|),$ $x \in \mathbb{R}^n$, with $g \downarrow$  on $\mathbb{R}_+$, implies 
 \begin{equation}
    \bar{\rho}( Uh^*) \leq C \,  \bar{\sigma}(h^*), \ \ \ h \in  L_{\bar{\sigma}}(\mathbb{R}_+). 
\end{equation}
\end{theorem}

In the next section we provide background material on r.i. norms. As mentioned above, certain so-called interpolation spaces will be considered.

Section 3 and 4 are devoted, respectively, to the proof of \Cref{FT implies U} and Theorem 4.1, the latter of which gives a general mapping theorem 
for $\mathscr{F}$ between r.i. spaces. Section 5 deals with $\mathscr{F}$
in the context of Orlicz spaces and Section 6 with $\mathscr{F}$ between
Lorentz Gamma spaces. Finally, in Section 7 we briefly discuss recent work on weighted norm inequalities for $\mathscr{F}(f)$.

Throughout this article, we write  $A \simeq B $ to abbreviate $C_1 A\leq  B \leq C_2 A$ for some constants $C_1, C_2 > 0 $ independent of $A$ and $B$.

\section{Rearrangement invariant spaces}\label{RI.Spaces}
\medskip
A theorem of Hardy and Littlewood asserts that \begin{equation}
\int_{\mathbb{R}^n}f(x)g(x)dx \leq \int_{\mathbb{R_+}}f^*(t)g^*(t)dt, \ \ \ f,g \in M_+(\mathbb{R}^n).
\end{equation}
The operation of rearrangement, though not sublinear itself, is sublinear in the average, namely,
\begin{equation}
(f+g)^{**}(t) \leq f^{**}(t) + g^{**}(t), \ \ \ f,g \in M_+(\mathbb{R}^n), \ \ t \in \mathbb{R_+}.
\end{equation}

\begin{definition}
\label{defn-ri-spaces}
A rearrangement-invariant (r.i.) Banach function norm $\rho$ on $M_+(\Omega )$, $\Omega= \mathbb{R}^n$ or $\mathbb{R_+}$,
satisfies
\begin{enumerate}
\item $\rho(f) \geq 0$, with $\rho(f)=0$ if and only if $f=0$ a.e.;
\item $\rho(cf)=c \, \rho(f)$, $c>0$;
\item $\rho(f+g) \leq \rho(f) + \rho(g)$;
\item $0 \leq f_n \uparrow f$ implies $\rho(f_n) \uparrow \rho(f)$;
\item $\rho(\chi_{E})<\infty$ for all measurable $E \subset \Omega$ such that $|E|<\infty$;
\item $\int_{E} f \leq C_{E} \, \rho(f)$, with $E \subset \Omega$, $|E| < \infty$ and $C_{E}>0$ independent of $f \in M_+(\Omega)$;
\item $\rho(f)= \rho(g)$ whenever $\mu_{f} = \mu_{g}$.
\end{enumerate}
\end{definition}

According to a fundamental result of Luxemberg \cite[Chapter 2, Theorem 4.10]{BS88}, there corresponds to every r.i. norm $\rho$ on $M_+(\mathbb{R}^n)$ an r.i. norm $\bar{\rho}$ on $M_+(\mathbb{R_+})$ such that
\begin{equation}
    \rho(f) = \bar{\rho}(f^{*}), \ \ \ f \in M_+(\mathbb{R}^n).
\end{equation}
A basic technique for working with r.i. norms involves the Hardy-Littlewood-Polya (HLP) Principle which asserts that
\begin{equation*}
    f^{**} \leq g^{**} \ \ \ \text{implies} \ \ \ \rho(f) \leq \rho(g);
\end{equation*}
see \cite[Chapter 3, Proposition 4.6]{BS88}. This principle
is based on a result of Hardy, a generalized form of which reads
\begin{equation}
    \int_{0}^{t} f \leq \int_{0}^{t} g
\end{equation}
implies  
\begin{equation*}
    \int_{0}^{t} f h^{*} \leq \int_{0}^{t} g h^{*}, \ \ \ f,g,h \in M_+(\mathbb{R_+}), \ \ \ t \in \mathbb{R_+}.
\end{equation*}
  
The K\"othe dual of an r.i. norm $\rho$ on $M_+(\Omega)$ is another such norm, $\rho'$, with
 \begin{equation*}
\rho'(g):= \sup_{\rho(f) \leq 1} \int_{\Omega} f(x)g(x), \ \ \   f, g \in M_+(\Omega).
\end{equation*}
 It obeys the Principle of Duality
\begin{equation}
    \rho'' = (\rho')' = \rho.
\end{equation}
 Further, one has the H\"older inequality
 \begin{equation*}
     \int_{\Omega} fg \leq \rho(f) \rho'(g),  \ \ \ f, g \in M_+(\Omega).
 \end{equation*}
 Finally,
 \begin{equation*}
     \bar{\rho'} = (\bar{\rho})'.
 \end{equation*}

Corresponding to an r.i. norm $\rho$ on $M_+(\Omega)$ is the class \begin{equation*}
L_{\rho}(X):= \left\lbrace f \in M(\Omega) : \rho(|f|)< \infty \right\rbrace,
\end{equation*}
which becomes a Banach space under the norm $\| f \|_{\rho} = \rho (|f|)$, $f \in L_{\rho}(\Omega)$. The space $L_{\rho}(\Omega)$ is then a rearrangement-invariant space.

The Orlicz and Lorentz Gamma spaces studied in sections $5$ and $6$, respectively, are examples of such r.i. spaces.

The dilation operator $E_{s}$, $s \in \mathbb{R_+}$, is defined at $f \in M_+(\mathbb{R_+})$, $t \in \mathbb{R_+}$, by
\begin{equation*}
(E_s f)(t)= f(st), \ \ \ s, t \in \mathbb{R_+}.
\end{equation*}
The operator $E_{s}$ is bounded on any r.i. space $L_{\rho}(\mathbb{R_+})$. We denote its norm by $h_{\rho}(s)$. Using $h_{\rho}$ we define the lower and upper  indices of $L_{\rho}(\mathbb{R_+})$ as
\begin{equation}\label{Boyd indices def}
i_{\rho}= \sup\limits_{s > 1} \frac{- \log h_{\rho}(s)}{\log s}  \ \ \  \text{and}  \ \ \ I_{\rho}= \inf\limits_{0<s < 1} \frac{- \log h_{\rho}(s)}{\log s},
\end{equation}
respectively. One has
\begin{equation*}
    i_{\rho} = \lim\limits_{s \rightarrow \infty} \frac{- \log h_{\rho}(s)}{\log s} \ \ \ \text{and} \ \ \ I_{\rho}= \lim\limits_{s \rightarrow 0^+} \frac{- \log h_{\rho}(s)}{\log s}.
\end{equation*}
Further, $0 \leq i_{\rho} \leq I_{\rho} \leq 1$ and, moreover,
\begin{equation*}
    i_{\rho'}= 1 - I_{\rho} \ \ \ \text{and} \ \ \ I_{\rho'} = 1 - i_{\rho}.
\end{equation*}
For all this, see \cite[pp. 1250-1252]{Bo69}.

If we denote by $k_{\rho}(s)$ the norm of $E_{s}$ on the characteristic functions $\chi_{F}$, $F \subset \mathbb{R_+}$, $|F|< \infty$, and define $j_{\rho}$ and $J_{\rho}$ by replacing $h_{\rho}(s)$ in (\ref{Boyd indices def}) by $k_{\rho}(s)$, we obtain the fundamental indices of $L_{\rho}(\mathbb{R_+})$. It turns out that when $L_{\rho}(\mathbb{R_+})$ is an Orlicz space or Lorentz Gamma space $i_{\rho}= j_{\rho}$ and $I_{\rho}= J_{\rho}$.

Finally, we describe that part of Interpolation Theory which is relevant to this paper.

Let $X_{1}$ and $X_{2}$ be Banach spaces compatible in the sense that both are continuously imbedded in the same Hausdorff topological space $H$, written
\begin{equation*}
    X_{i} \hookrightarrow H, \ \ \ i = 1,2.
\end{equation*}
The spaces $X_{1} \cap X_{2}$ and $X_{1} + X_{2}$ are the sets
\begin{equation*}
    X_{1} \cap X_{2} := \left\lbrace x : x \in X_{1} \ \ \text{and} \ \ x \in X_{2} \right\rbrace
\end{equation*}
and
\begin{equation*}
    X_{1} + X_{2} := \left\lbrace x : x= x_{1} + x_{2}, \ \ \text{for some} \  x_{1} \in X_{1},  \  x_{2} \in X_{2} \right\rbrace,
\end{equation*}
with norms
\begin{equation*}
    \|x \|_{X_{1} \cap X_{2}  } = \max \left[ \,  \|x \|_{X_{1}} , \|x \|_{X_{2}} \right]
\end{equation*}
and
\begin{equation*}
    \|x \|_{X_{1} + X_{2}  } = \inf \left\lbrace \,  \|x_{1} \|_{X_{1}} + \|x_{2} \|_{X_{2}} : x= x_{1} + x_{2},  \,  x_{1} \in X_{1},  \  x_{2} \in X_{2} \right\rbrace.
\end{equation*}

Recall that given Banach spaces $X_{1}$ and $X_{2}$ imbedded in a common Hausdorff topological vector space, their Peetre $K$-functional is defined for $x \in X_{1}+ X_{2}$, $t>0$, by
\begin{equation*}
    K(t,x;X_{1},X_{2}) = \inf_{x \, = \, x_{1} + x_{2}} \left[ \| x_{1} \|_{X_{1}}  +  t \,  \| x_{2} \|_{X_{2}}  \right].
\end{equation*}

We observe that, for $\Omega = \mathbb{R}^n$ or $\mathbb{R_+}$, $p \in [1, \infty)$, $L_{p}(\Omega)$ and $L_{\infty}(\Omega)$ are compatible, each being continuously imbedded in the Hausdorff topological space $M(\Omega)$ equipped with the topology of convergence in measure. One has
\begin{equation*}
    K(t, f; L_{p}(\Omega),  L_{\infty}(\Omega)) \simeq \left[ \int_{0}^{t^{p}} f^{*}(s)^{p} ds \right]^{1/p}, \   t>0,
\end{equation*}
$f \in \left( L_{p} + L_{\infty} \right)(\Omega)$. 

In the above notation, the   inequality (\ref{JT U dominate F}) from \cite{JT71} reads
\begin{equation} \label{JT interms of K functional}
    K \left(t, (\hat{f})^{*}; L_{2}( \mathbb{R_+} ),  L_{\infty}( \mathbb{R_+} ) \right) \leq  K \left(t, C^{\frac{1}{2}} \, U f^{*}; L_{2}( \mathbb{R_+} ),  L_{\infty}( \mathbb{R_+} ) \right).
\end{equation}

\begin{definition} A Banach space $Y$ is said to be intermediate between $X_{1}$ and $X_{2}$ if
\begin{equation*}
    X_{1} \cap X_{2} \hookrightarrow  Y  \hookrightarrow X_{1} + X_{2}. 
\end{equation*}
\end{definition}

\begin{definition} A Banach space $Y$ intermediate between the compatible spaces $X_{1}$ and $X_{2}$ is said to be monotone if, given $x, y \in X_{1} + X_{2}$, with
\begin{equation} \label{monotone wrt X1 X2}
    K(t, x; X_{1}, X_{2}) \leq K(t, y; X_{1}, X_{2}), \ \ \ t \in \mathbb{R_+},
\end{equation}
one has $y \in Y$ implies $x \in Y$ and $ \| x \|_{Y} \leq  \| y \|_{Y}$.
\end{definition}

The result of Lorentz-Shimogaki in \cite[Theorem 2 and Lemma 3]{LS71} asserts that the r.i. interpolation  spaces between $L_{p}(\Omega)$ and $L_{\infty}(\Omega)$ are precisely the monotone spaces in that context. Further, the inequality (\ref{JT interms of K functional}) is a special case of (\ref{monotone wrt X1 X2}). Thus, for $L_{\rho}(\mathbb{R}^{n})$ between $L_{2}(\mathbb{R}^{n})$ and $L_{\infty}(\mathbb{R}^{n})$, there holds
\begin{equation*}
    \begin{split}
        \rho ( \lvert \hat{f}   \rvert) = \bar{\rho} \left(  (\hat{f})^{*} \right) 
        & \leq C^{\frac{1}{2}} \bar{\rho} \left( U f^{*} \right)\\
        & \leq M C^{\frac{1}{2}} \bar{\sigma} (  f^{*} )\\
         & = M C^{\frac{1}{2}} \sigma (  f ),
        \end{split}
\end{equation*}
whenever the r.i. norms $\rho$ and $\sigma$ on $M_{+}(\mathbb{R}^{n})$ satisfy
 \begin{equation}
    \bar{\rho}( U f^{*}) \leq M \,  \bar{\sigma}(f^{*}), \ \ \ f \in  M_{+}(\mathbb{R}^{n}). 
\end{equation}

\section{Proof of \texorpdfstring{\Cref{FT implies U}}{TEXT} }

Let $f \in (L_{1} \cap L_{\sigma}  )(\mathbb{R}^n) $ be such that $f(x)= g(|x|)$, $x \in \mathbb{R}^n$, with $g \downarrow$ on $\mathbb{R_+}$. Denote by $B_{r}$ the ball of radius $r>0$ about the origin in $\mathbb{R}^n$. Set $r_{n}=r_{n}(t)= (\nu_{n}t)^{- 1 / n}$, $t \in \mathbb{R_+}$, $\nu_{n} = |B_{1}|= \pi^{\frac{n}{2}} / \Gamma(\frac{n}{2} + 1)$.

Now, 
\begin{equation*}
    \widehat{ \chi_{B_{r}}} (\xi) =  r^{n} \, \frac{ J_{\frac{n}{2}  } (2\pi r |\xi|) }{ (r |\xi|)^{\frac{n}{2}}  }, \ \ \ \xi \in \mathbb{R}^n,
\end{equation*}
where $J_{\frac{n}{2}}$ is the Bessel function of the first kind of order $\frac{n}{2}$.

Thus, with $f,t$ and $r_{n}$ as above, we have
\begin{equation*}
    \begin{split}
        (U f^{*})(t) 
        & = \int_{0}^{1/t} f^{*} = \int_{B_{r_{n}}} f = (4r_{n})^{n} \int_{ B_{ 1 /4 } } f( 4 r_{n} y) dy\\
        & \leq (4 r_{n})^{n} C_{n} \int_{B_{1 / 4} } \widetilde{J_{\frac{n}{2}}} (2 \pi |y|)^{2}  f( 4 r_{n} y) dy\\
        & \leq  (4 r_{n})^{n} C_{n}  \int_{\mathbb{R}^n} \widetilde{J_{\frac{n}{2}}} (2 \pi |y|)^{2}  f( 4 r_{n} y) dy,
    \end{split}
\end{equation*}
here, $\widetilde{ J_{\frac{n}{2}}   } (s) = \frac{  J_{\frac{n}{2}} (s) }{ s^{\frac{n}{2}}}$, $s \in \mathbb{R_+}$. The last term equals
\begin{equation*}
    \begin{split}
        & 4^{n} C_{n} \int_{\mathbb{R}^n} \reallywidehat{ \left(  \widetilde{J_{\frac{n}{2}}} (2 \pi \cdot) \right)^{2} } ( \xi)  \hat{f}( {\textstyle\frac{\xi}{ 4 r_{n}} } ) \, d\xi \\
        & =  4^{n} (2 \pi )^{-n} C_{n} \int_{\mathbb{R}^n} \left( \chi_{B_{1}} \ast \chi_{B_{1}}  \right) ( \xi)   \hat{f}( {\textstyle\frac{\xi}{ 4 r_{n}} } ) \, d\xi \\
        & \leq  ( { \textstyle\frac{\pi}{2} } )^{-n} C_{n} \int_{2 B_{1}} \left|  \hat{f}( {\textstyle\frac{\xi}{ 4 r_{n}} } ) \right| \, d\xi \\
        & = ( { \textstyle\frac{8}{\pi} } )^{n} C_{n} \,  r_{n}^{n} \int_{|\xi| \leq (2 \,  r_{n})^{-1}  } \left| \hat{f}( \xi) \right| \, d\xi, \\
        & \leq  ( { \textstyle\frac{8}{\pi} } )^{n} C_{n} \,  r_{n}^{n} \int_{0}^{ 2^{-n} \nu_{n}^{2}  t} (\hat{f})^{*}\\
         & =  ( { \textstyle\frac{8}{\pi} } )^{n} \frac{C_{n}}{ \nu_{n}}  \,  \frac{1}{t} \int_{0}^{ 2^{-n} \nu_{n}^{2}  t} (\hat{f})^{*}\\
          & \leq  C_{n}' \,   \frac{1}{t} \int_{0}^{ t} (\hat{f})^{*},
    \end{split}
\end{equation*}
that is,
\begin{equation*}
    \int_{0}^{t^{-1}} f^{*} \leq C_{n}' \frac{1}{t} \int_{0}^{t} (\hat{f})^{*},
\end{equation*}
where $C_{n}' = ( { \textstyle\frac{8}{\pi} } )^{n} \frac{C_{n}}{ \nu_{n}} \max [1, 2^{-n} \nu_{n}^2 ]$.

Hence,
\begin{equation*}
    \begin{split}
        \bar{\rho} (U f^{*}) 
        & \leq C_{n}' \,  \bar{\rho} \left(  \frac{1}{t} \int_{0}^{t} (\hat{f})^{*} \right)\\
        & \leq C_{n}' \,  \bar{\rho} \left(   (\hat{f})^{*} \right), \ \  \text{by} \ (\ref{P boundedness on decreasing}), \\
        & =  C_{n}' \, \rho  (\hat{f})\\
        & \leq C \,  C_{n}' \, \sigma (f), \ \  \text{by assumption,}\\
        & = C \,  C_{n}' \,  \bar{\sigma}(f^{*}).
    \end{split}
\end{equation*}

Since, for $g \in M_+(\mathbb{R_+})$,
\begin{equation*}
    U g \leq U \left( g^{*} \right),
\end{equation*}
we get
\begin{equation*}
    \bar{\rho} (U g) \leq \bar{\rho} \left(  U  g^{*} \right) \leq C \bar{\sigma} (g). \ \ \  \square
\end{equation*}

\section{\texorpdfstring{$\mathscr{F}$}{TEXT}  between r.i. spaces} 

\begin{theorem}\label{F implies U on riBFS}
Let $\rho$ and $\sigma$ be r.i. norms on $M_+(\mathbb{R}^n)$ determined, respectively, by the r.i. norms $\bar{\rho}$ and $\bar{\sigma}$ on $M_+(\mathbb{R_+})$. Assume $L_{\bar{\rho}}(\mathbb{R_+})$ is an interpolation space between $L_{2}(\mathbb{R_+})$ and $L_{\infty}(\mathbb{R_+})$, so that, in particular, (\ref{P boundedness on decreasing}) holds.
Then,
\begin{equation*}
    \rho( \lvert \hat{f} \rvert ) \leq C \sigma(f), \ \ \ f \in (L_{\sigma} \,  \cap \, L_{1})(\mathbb{R}^n),
\end{equation*}
if and only if
\begin{equation*}
    \bar{\rho} (U g) \leq C \bar{\sigma} (g), \ \ \  g \in  L_{\bar{\sigma}}(\mathbb{R_+}).
\end{equation*}
\end{theorem}
\begin{proof}
The only if part is the content of \Cref{FT implies U}. Again, the if part was proved at the end of the Section 2.
\end{proof}
Boyd in \cite[pp. 92-98]{Bo67}  associates to each r.i. norm  $\rho$ on $M_+(\Omega)$ and each $p>1$ the functional
\begin{equation}\label{Rho-p}
    \rho^{(p)} (f) = \rho (f^{p})^{\frac{1}{p}}, \ \ \ f \in M_+(\Omega).
\end{equation}
He shows that $\rho^{(p)}$ is an r.i. norm on $M_+(\Omega)$ and that (\ref{P boundedness on decreasing}) holds with $\rho = \rho^{(p)}$.

\begin{theorem} \label{Int Space bw p and infinity}
Let $\rho$  be an r.i. norm on $M_+(\Omega)$. For fixed $p>1$, define $\rho^{(p)}$ as in (\ref{Rho-p}). Then, $L_{\rho^{(p)}}(\Omega)$ is an interpolation space between $L_{p}(\Omega)$ and $L_{\infty}(\Omega)$.
\end{theorem}
\begin{proof}
Suppose the linear operator $T$ satisfies
\begin{equation*}
    T : L_{p}(\Omega) \rightarrow L_{p}(\Omega) \ \ \ \text{and} \ \ \ T : L_{\infty}(\Omega) \rightarrow L_{\infty}(\Omega).
\end{equation*}
Then, according to  \cite[Theorem 1.11, pp. 301-304]{BS88}, there exists $C>0$, such that 
\begin{equation}
    \int_{0}^{t} (Tf)^{*}(s)^{p} ds \leq C \int_{0}^{t} f^{*}(s)^{p} ds, \ \ \ f \in (L_{p} + L_{\infty})(\Omega), \ t \in \mathbb{R_+}.
\end{equation}

The HLP Principle involving $\bar{\rho}$ yields
\begin{equation*}
    \begin{split}
        \rho \left( (Tf)^{p} \right) & = \bar{\rho} \left( \left[ (Tf)^{*}  \right]^{p} \right) \\
        & \leq C   \bar{\rho} \left( \left[ f^{*}  \right]^{p} \right)\\
        & = C \rho \left(  |f|^{p} \right), \ \ \ f \in L_{\rho^{(p)}}(\Omega),
    \end{split}
\end{equation*}

and hence
\begin{equation*}
   \rho^{(p)} (Tf) \leq C \rho^{(p)} (f), \ \ \ f \in L_{\rho^{(p)}}(\Omega).
\end{equation*}
\end{proof}

\begin{theorem}
Let $\rho$ and $\sigma$ be r.i. norms on $M_+(\mathbb{R}^n)$ determined, respectively, by the r.i. norms $\bar{\rho}$ and $\bar{\sigma}$ on $M_+(\mathbb{R_+})$ by  $\rho(f) = \bar{\rho}(f^{*})$ and $\sigma(f) = \bar{\sigma}(f^{*})$, $f \in M_+(\mathbb{R}^n)$. Then,
\begin{equation*}
    \rho^{(2)} ( \hat{f}) \leq C \sigma(f), \ \ \ f \in (L_{\sigma} \,  \cap \, L_{1})(\mathbb{R}^n),
\end{equation*}
if and only if
\begin{equation*}
     \overline{\rho^{(2)}} (Ug)  \leq C \bar{\sigma} (g), \ \ \ g \in  M_+(\mathbb{R_+}).
\end{equation*}
\end{theorem}
\begin{proof}
The result is a consequence of \Cref{F implies U on riBFS} and \Cref{Int Space bw p and infinity}.
\end{proof}

Suppose now that $\mu$ is an r.i. norm on $M_+(\mathbb{R_+})$ satisfying $\mu \left( \frac{1}{1+t} \right) < \infty$. Denote by $X_{\mu}$ the set of all $x \in X_{1}+ X_{2}$ for which
\begin{equation*}
    \rho_{\mu} (x)= \mu \left( \frac{ K(t,x;X_{1},X_{2})  }{t}  \right) < \infty.
\end{equation*}
Then, $X_{\mu}$ with the norm $\rho_{\mu}$ is an interpolation space between $X_{1}$ and $X_{2}$.

It is well known that for $f \in \left( L_{p} + L_{\infty} \right)(\Omega)$, $p \geq 1$ one has
\begin{equation*}
    K(t, f; L_{p}(\Omega),  L_{\infty}(\Omega)) \simeq \left[ \int_{0}^{t^{p}} f^{*}(s)^{p} ds \right]^{1/p}, t>0,
\end{equation*}
so the space $X_{\rho_{\mu, p}}$, with the norm
\begin{equation*}
    \rho_{\mu, p} (f) = \overline{\rho_{\mu, p}} (f^{*}) = \bar{\mu } \left(  t^{-1} \left[ \int_{0}^{t^{p}} f^{*}(s)^{p} ds \right]^{1/p}  \right),   \ \ \  f \in M_+(\Omega),
\end{equation*}
is an interpolation space between $L_{p}(\Omega)$ and $L_{\infty}(\Omega)$. In view of this, \Cref{F implies U on riBFS} guarantees

\begin{theorem}
Let $\mu$ and $\sigma$ be r.i. norms on $M_+(\mathbb{R}^n)$ determined, respectively, by the r.i. norms $\bar{\mu}$ and $\bar{\sigma}$ on $M_+(\mathbb{R_+})$. Suppose 
$\mu \left( \frac{1}{1+t} \right) < \infty$. Set
\begin{equation*}
    \rho_{\mu, 2} (f) = \overline{\rho_{\mu, 2}} (f^{*}) = \bar{\mu} \left(  t^{-1} \left[  \int_{0}^{t^{2}} f^{*}(s)^{2} ds \right]^{1/2}  \right).
\end{equation*}
Then,
\begin{equation*}
      \rho_{\mu, 2} ( \hat{f}) \leq C \sigma(f), \ \ \ f \in (L_{\sigma} \,  \cap \, L_{1})(\mathbb{R}^n),
\end{equation*}
if and only if
\begin{equation*}
    \overline{  \rho_{\mu, 2} } (Ug)  \leq C \bar{\sigma} (g), \ \ \ g \in  M_+(\mathbb{R_+}).
\end{equation*}
\end{theorem}

Finally, consider an r.i. norm $\rho$ on $M_+(\mathbb{R}^n)$ determined by the r.i. norm $\Bar{\rho}$ on $M_+(\mathbb{R_+})$ and set
\begin{equation*}
    \rho_{U} (f) := \left( \Bar{\rho} \,  \circ  U \right)(f^{*}) = \bar{\rho} (U f^{*}), \ \ \ f \in M_+(\mathbb{R}^n).
\end{equation*}
One has $\rho_{U}$  an r.i. norm if $\left( \Bar{\rho} \,  \circ  U \right)( \chi_{(0,t)} ) < \infty$ for all $t>0$, or, equivalently, $\bar{\rho} \left( \frac{1}{1+t} \right)  < \infty$. In that case, $L_{\bar{\rho} \, \circ U  } (\mathbb{R_+})$ is the largest r.i. space to be mapped into $L_{\bar{\rho}}(\mathbb{R_+})$ by $U$.

With this background we now have
\begin{theorem}
Let $\rho$ be an r.i. norm on $M_+(\mathbb{R}^n)$ defined in terms of an r.i. norm $\bar{\rho}$ on $M_{+}(\mathbb{R_+})$ such that
\begin{equation*}
    \bar{\rho} \left( \frac{1}{1+t} \right)  < \infty.
\end{equation*}
Assuming $L_{ \bar{\rho}}(\mathbb{R_+})$ is an interpolation space between $L_{2}(\mathbb{R_+})$ and $L_{\infty}(\mathbb{R_+})$, one has that $L_{\rho_{U}} (\mathbb{R}^n)$ is the largest r.i. space of functions on $\mathbb{R}^n$ to be mapped into $L_{\rho}(\mathbb{R}^n)$ by $\mathscr{F}$.
\end{theorem}

\section{\texorpdfstring{$\mathscr{F}$}{TEXT} in the context of Orlicz spaces} 

An Orlicz gauge norm is given in terms of an $N$-function
\begin{equation*}
    \Phi(x)= \int_{0}^{x} \phi, \ \ \ x \in \mathbb{R_+};
\end{equation*}
here $\phi$ is an nondecreasing function mapping $\mathbb{R_+}$ onto itself. Specifically, the gauge norm $\rho_{\Phi}$ is defined at $f \in M_+(\Omega)$ by
\begin{equation*}
   \rho_{\Phi}(f) = \inf \left\lbrace  \lambda >0 : \int_{\Omega} \Phi \left( \frac{f(x)}{\lambda} \right)  \leq 1   \right\rbrace,
\end{equation*}
One can show $\rho_{\Phi}(f)= \bar{\rho}_{\Phi}(f^{*})$, so that the Orlicz space
\begin{equation*}
    L_{\Phi}(\Omega)= \left\lbrace f \in M(\Omega) :  \rho_{\Phi}(|f|) < \infty    \right\rbrace
\end{equation*}
is an r.i. space.

The definitive work on $\mathscr{F}$ between Orlicz spaces is due to Jodiet and Torchinsky. See, in particular, \cite[Theorem 2.16]{JT71}. This theorem asserts that if $A$ and $B$ are $N$-functions with $L_{A}(\mathbb{R}^n) \subset  (L_{1} + L_{2})(\mathbb{R}^n)$, $L_{B}(\mathbb{R}^n) \subset  (L_{2} + L_{\infty})(\mathbb{R}^n)$ and $\mathscr{F} : L_{A}(\mathbb{R}^n) \rightarrow L_{B}(\mathbb{R}^n)$, then there exist $N$-functions $A_{1}$ and $B_{1}$ with $L_{A_{1}}(\mathbb{R}^n) \supset L_{A}(\mathbb{R}^n)$ and $L_{B_{1}}(\mathbb{R}^n) \subset L_{B}(\mathbb{R}^n)$ for which $\mathscr{F} : L_{A_{1}}(\mathbb{R}^n) \rightarrow L_{B_{1}}(\mathbb{R}^n)$. Moreover, $B_{1}(t)= 1 / \tilde{A_{1}}(t^{-1})$, $A_{1}(t) / t^{2} \downarrow$ on $\mathbb{R_+}$ and so $B_{1}(t) / t^{2}$ $\uparrow$ on $\mathbb{R_+}$.
 (The function $\tilde{A_{1}}$ can be defined through the equation $ \tilde{A_{1}}^{-1} (t) = \frac{t}{ A_{1}^{-1}(t)}, t \in \mathbb{R_+} $)

Using the results in the previous sections we now show $L_{A_{1}}(\mathbb{R_+})$ is an interpolation space between $L_{1}(\mathbb{R_+})$ and $L_{2}(\mathbb{R_+})$, while $L_{B_{1}}(\mathbb{R_+})$ is an interpolation space between $L_{2}(\mathbb{R_+})$ and $L_{\infty}(\mathbb{R_+})$.

To begin, we observe that $B_{1}(t) / t^{2}$ $\uparrow$ is equivalent to $B_{1}(t) = \Phi(t^{2})$ for some $N$- function $\Phi$.  Indeed, given the later, one has $B_{1}(t) / t^{2} = \Phi(t^{2}) / t^{2} \uparrow$ on $\mathbb{R_+}$. Again $B_{1}(t) / t^{2}  \uparrow$ implies $B_{1}(t^{1/2} ) / t  \uparrow$, so that $\Phi(t)=B_{1}(t^{1/2})$ - which is equivalent to an $N$-function with derivative  $B_{1}(t^{1/2}) / t$ - is such that $B_{1}(t)= \Phi(t^{2})$.
Next, 
\begin{equation*}
    \rho_{\Phi}^{(2)} (f) = \rho_{\Phi} (f^{2})^{\frac{1}{2}}= \rho_{B_{1}} (f).
\end{equation*}
According to \Cref{Int Space bw p and infinity}, then,
\begin{equation*}
    \rho_{B_{1}} ( \lvert Tf \rvert ) = \rho_{2} ( \lvert Tf \rvert) \leq C \rho_{2} (f) = C \rho_{B_{1}} (f),
\end{equation*}
that is,
\begin{equation*}
    \rho_{B_{1}} ( \lvert Tf \rvert ) \leq C \rho_{B_{1}} (f), \ \ \ f \in L_{B_{1}}(\mathbb{R}^n),
\end{equation*}
whenever $T$ is a linear operator satisfying
$ T : L_{2}(\mathbb{R}^n) \rightarrow L_{2}(\mathbb{R}^n)$ and $ T : L_{\infty}(\mathbb{R}^n) \rightarrow L_{\infty}(\mathbb{R}^n)$. Thus, $L_{B_{1}}(\mathbb{R}^n)$ is an interpolation space between $L_{2}(\mathbb{R}^n)$ and $L_{\infty}(\mathbb{R}^n)$.

Now, $B_{1}(t)= 1 / \tilde{A_{1}}(t^{-1})$ is equivalent to $ \tilde{A_{1}}(t)= 1 / B_{1}(t^{-1})$, whence $\tilde{A_{1}}(t) / t^{2}$ $= (t^{-1})^{2} / B_{1}(t^{-1})$
and so $B_{1}(t) / t^{2} \uparrow$ amounts to $\tilde{A_{1}}(t) / t^{2} \uparrow$, that is, $L_{\tilde{A_{1}}}(\mathbb{R}^n)$ is an interpolation space between $L_{2}(\mathbb{R}^n)$ and $L_{\infty}(\mathbb{R}^n)$. Since $L_{\tilde{A_{1}}}(\mathbb{R}^n)$ is the Kothe dual of $L_{A_{1}}(\mathbb{R}^n)$ we conclude that $L_{A_{1}}(\mathbb{R}^n)$ is an interpolation space between 
$L_{1}(\mathbb{R}^n)$ and $L_{2}(\mathbb{R}^n)$.

The monotonicity conditions on $A_{1}$ and $B_{1}$ translate into conditions on their associated fundamental functions. For example, $L_{B_{1}}(\mathbb{R}^n)$ has fundamental function $\phi_{B_{1}}(t)= \rho_{B_{1}} \left( \chi_{(0,t)} \right) = 1 / B_{1}^{-1}(t^{-1})$, $t \in \mathbb{R_+}$. Thus setting $t = 1 / B_{1}(y)$ in $\frac{\phi_{B_{1}}(t)}{t^{1/2}} = \frac{1}{B_{1}^{-1}(t^{-1}) t^{1/2}}$ we arrive at $\left( \frac{B_{1}(y)}{ y^{2}} \right)^{1/2}$, which increases in $y$ and therefore decreases in $t$, so $\frac{\phi_{B_{1}} (t) }{ t^{1/2}} \downarrow$.

Lastly, we remark that $L_{A_{1}} (\mathbb{R}^n)$ is not the largest r.i. space that $\mathscr{F}$ maps into $L_{B_{1}} (\mathbb{R}^n)$; that space has norm $\rho_{B_{1}} (U f^{*})$. In the Lebesgue context, in which, say,  $B_{1}(t)= t^{p'}$, $1<p<2$,
\begin{equation*}
    \rho_{p'} (U f^{*}) = \left[ \int_{\mathbb{R_+}} (U f^{*})(t)^{p'} dt \right]^{1/p'}  \simeq \left[ \int_{\mathbb{R_+}} \left[ t^{\frac{1}{p}} f^{**}(t) \right]^{p'} \frac{dt}{t} \right]^{1/p'},
\end{equation*}
which is the so-called Lorentz norm $\rho_{p,p'}$. This norm is smaller than $\rho_{A_{1}} = \rho_{p}$. For more details see the next section.

\section{\texorpdfstring{$\mathscr{F}$}{TEXT} between Lorentz Gamma spaces} 
In this section, we make use of the operators $P$ and $Q$ defined by
\begin{equation*}
    (P f) (t) = \frac{1}{t} \int_{0}^{t} f \ \ \ \text{and} \ \ \ (Q g )(t) = \int_{t}^{\infty} g(s) \frac{ds}{s}, \ \ \ f,g \in M_+(\mathbb{R_+}), \ t \in \mathbb{R_+}.
\end{equation*}
These operators satisfy the equations
\begin{equation*}
    \int_{\mathbb{R_+}} g \, Pf = \int_{\mathbb{R_+}} f \, Qg, \ \ \ f,g \in M_+(\mathbb{R_+}),
\end{equation*}
and
\begin{equation*}
    PQ= QP = P + Q.
\end{equation*}
Fix $p,q \in (1,\infty)$ and $u,v \in M_+(\R_+)$. The Lorentz Gamma space
$\Gamma_{p,u}(\Omega)$ is defined as 

$$\Gamma_{p,u}(\Omega)=\left\{f \in  M(\Omega): \rho_{p,u}(f)= \left( \int_{\mathbb{R_+}} f^{**}(t)^p \, u(t) \,  dt \right)^{\frac{1}{p}} <\infty \right\},$$
with a similar definition for $\Gamma_{q,v}(\Omega)$.

To guarantee $\rho_{p,u}(\chi_{E}),$ $\rho_{q,v}(\chi_{E}) < \infty$ for all
$E \subset \Omega, |E| < \infty$ we require
\begin{equation}
     \int_{\mathbb{R_+}} \frac{u(t)}{(1+t)^p} dt, \ \ \  \int_{\mathbb{R_+}} \frac{v(t)}{(1+t)^q} dt <\infty.
    \end{equation}

In this section we study the inequality 
\begin{equation}\label{F on Lorentz-p-q}
    \rho_{p,u}(\hat{f}) \leq C \,  \rho_{q,v}(f),  \ \ \  f \in (L_1\cap \Gamma_{q,v})(\mathbb R^n).
    \end{equation}

We begin by assuming $\Gamma_{p,u}(\mathbb{R_+})$ is an interpolation space between $L_{2}(\mathbb{R_+})$ and $L_{\infty}(\mathbb{R_+}),$ then address the question of when this is the case later in the section.

Recall that Theorem 1.1 ensures that (6.2.) holds if and only if
\begin{equation}
    {\bar{\rho}}_ {p,u}(Uf^*)  \leq C {\bar{\rho}}_ {q,v}(f^*), \ \ \   f \in  M_+(\mathbb{R_+}).
\end{equation}

\begin{theorem}\label{Gamma bw L2 and L infinity}
Let the indices $p,q$ and weights $u,v$ be as described above. Then, given that 
 $\Gamma_{p,u}(\mathbb{R_+})$ is an interpolation space between $L_{2}(\mathbb{R_+})$ and $L_{\infty}(\mathbb{R_+}),$ one has (6.2) if and only
 if 
\begin{equation}\label{L2 Linfinity interpolation equiv to P}
    {\bar{\rho}}_{q',v'}(g^{**})  \leq C {\bar{\rho}}_{p',{u_p}'}(g^{*}), \ \ \
     g \in  M_+(\mathbb{R_+}).
\end{equation}
 where, as usual, $p'=\frac{p}{p-1}$, $ q'=\frac{q}{q-1}$, $\int_{\mathbb{R_+}} v=\infty$,

 \begin{equation*}
    v'(t)= \frac{t^{q'+q-1}\int_0^t v\int_t^\infty v(s)s^{-q} ds} {{\left[\int_0^t v+t^q \int_t^\infty v(s)s^{-q} ds\right]}^{q'+1}}
\end{equation*}

\begin{equation*}
   u_p(t)= u(t^{-1})t^{p-2}, \ \ \  \int_{\mathbb{R_+}} u_{p} = \infty , 
\end{equation*}
and 
\begin{equation*}
    {u_p}'(t)= \frac{  t^{p'+p-1}  \int_{0}^{t} u_{p}  \int_{t}^{\infty} u_{p}(s) s^{-p} ds }   {  \left[ \int_{0}^{t} u_{p} + t^{p} \int_{t}^{\infty} u_{p}(s)s^{-p} ds \right]^{p'+1      } }, \ \ \  t \in {\R}_{+}.  
\end{equation*}
\end{theorem}
\begin{proof}
The inequality (6.3) tells that the space determined by the r.i. norm
${\bar{\rho}}_{p,u} (Uf^*)$ is the largest one mapped into $\Gamma_{p,u}(\R_+)$
by $U.$
Now,
\begin{equation*}
    \bar{\rho}_{p,u}(Uf^*)= \left[ \int_{ \mathbb{R}_+ } P(Uf^{*})(t)^{p} u(t) \, dt \right]^{\frac{1}{p}},
\end{equation*}

\begin{equation*}
    \begin{split}
        P \left( Uf^* \right) (t) & = t^{-1}   \int_0^t \int_0^{s^{-1}} f^*(y) dy \, ds \\
        & =  t^{-1}\int_{t^{-1}}^\infty \int_0^s f^*(y) dy \,  \frac{ds}{s^2} \\
        & =  t^{-1} \int_{t^{-1}}^{\infty } \frac{1}{s} \int_{0}^{s} f^{*}(y)  \, dy \frac{ds}{s} \\
        & =  t^{-1} \left[ t \int_{0}^{t^{-1}} f^{*} +  \int_{t^{-1}}^{\infty} f^*(y) \frac{dy}{y} \right]\\
        & = \int_{0}^{t^{-1}} f^{*} + t^{-1} \int_{t^{-1}}^{\infty} f^{*}(y) \frac{dy}{y}\\
        & = t^{-1} (P f^{*})( t^{-1}) + t^{-1} (Q f^{*})( t^{-1})\\
        & = t^{-1} \left[ (P+ Q ) f^{*} \right] ( t^{-1})\\
        & = t^{-1} \left( P (Q  f^{*}) \right) ( t^{-1})
    \end{split}
\end{equation*}

Further, 
\begin{equation*}
    \begin{split}
        \int_{\mathbb{R_+}}  \left[ t^{-1} \left( P (Q  f^{*}) \right)( t^{-1}) \right]^{p} u(t) dt 
        & = \int_{\mathbb{R_+}}  \left[ t  \left( P (Q  f^{*}) \right)( t) \right]^{p} u(t^{-1}) t^{-2} dt.
        \end{split}
\end{equation*}

 We have shown that 
\begin{equation*}
    \begin{split}
 {\bar{\rho}}_{p,u}(Uf^*) 
 &=\left[ \int_{\mathbb{R_+}} \left( P (Q  f^{*}) \right) (t)^p u_p(t) dt\right]^{\frac{1}{p}}\\
 &={\bar{\rho}}_{p,u_p}(Qf^*).
 \end{split}
\end{equation*}

Therefore, any $ \Gamma_{q,v}(\mathbb{R_+})$ mapped into  $ \Gamma_{p,u}(\mathbb{R_+})$
by $U$ must be embedded into this largest domain ; that is,

\begin{equation} \label{Operator Q L-Gamma-norm ineq}
    {\bar{\rho}}_ {p,u_p}(Qf^*) \leq C {\bar{\rho}}_ {q,v}(f^*) , \ \ \ f \in M_+(\R_+).
\end{equation}

But (\ref{Operator Q L-Gamma-norm ineq}) is equivalent to (\ref{L2 Linfinity interpolation equiv to P}), its dual inequality, (\ref{Operator Q L-Gamma-norm ineq}) may be tested over any $f \in M_+(\R_+),$ as is seen in 
\begin{equation*}
     \int_{\mathbb{R_+}} g^{*} Qf 
     =  \int_{\mathbb{R_+}} f P g^{*} \leq  \int_{\mathbb{R_+}} f^{*} P g^{*}= \int_{\mathbb{R_+}} g^{*} Q f^{*}.
     \end{equation*}

\end{proof}

The inequality (\ref{L2 Linfinity interpolation equiv to P}), and hence (\ref{F on Lorentz-p-q}), amounts to 

\begin{equation} \label{P from Lp' to Lq'}
\left[\int_{\mathbb{R_+}}   (Ph)^{q'}v' \right]^{\frac{1}{q'}} \leq C \left[\int_{\mathbb{R_+}} h^{p'}{u_p}' \right]^{\frac{1}{p'}},
\end{equation}

with $h=g^{**}$ belonging to
\begin{equation*}
     \Omega_{0,1}(\mathbb{R_+})= \{h \in M_+(\mathbb{R_+}): h(t) \downarrow
      \ \text{and} \  t \, h(t) \uparrow    \text{on}  \  \mathbb{R_+}\}.
     \end{equation*}
Such inequalities are shown in Theorem 4.4 of \cite{GK14} to be equivalent to a pair of weighted norm inequalities involving general non-negative measurable functions.
 In the case of (6.6) this leads to

 \begin{theorem}
Let the indices $p,q$ and weights $u,v$ be as in \Cref{Gamma bw L2 and L infinity}. Then, (6.6) holds if and only if
\begin{equation} \label{on nonnegative from P con cone}
 \left[ \int_{\mathbb{R_+}} \left[(P+Q)g \right]^{q'} v'\right]^{\frac{1}{q'}} \leq C
 \left[ \int_{\mathbb{R_+}} g^{p'}{u_p}^{1-p'}\right]^{\frac{1}{p'}}
\end{equation}

and

\begin{equation*}
 \left[ \int_{\mathbb{R_+}} (P^2g)^{q'} v'\right]^{\frac{1}{q'}} \leq C
 \left[ \int_{\mathbb{R_+}} g^{p'}{u_p}^{1-p'}\right]^{\frac{1}{p'}}
\end{equation*}

$ g \in M_+( \mathbb{R_+})$.

\end{theorem}

\begin{proof} According to Theorem 4.4 in \cite{GK14}, one has (\ref{P from Lp' to Lq'}) if and only if

\begin{equation*}
 \left[ \int_{\mathbb{R_+}} \left[(P+Q)Qg \right]^{p} u_p\right]^{\frac{1}{p}} \leq C
 \left[ \int_{\mathbb{R_+}} g^{q}{v'}^{1-q}\right]^{\frac{1}{q}}
\end{equation*}

and
\begin{equation*}
 \left[ \int_{\mathbb{R_+}} \left[ P(P+Q)g \right]^{q'} v'\right]^{\frac{1}{q'}} \leq C
 \left[ \int_{\mathbb{R_+}} g^{p'}{u_p}^{1-p'}\right]^{\frac{1}{p'}}
\end{equation*}

$ g \in M_+(\mathbb{R_+}).$

These are dual inequalities. We choose the second one, which easily reduces to (\ref{on nonnegative from P con cone}).

\end{proof}

To deal with the case $q \leq p$ we will use special instances of the following combination of Theorem 1.7 and 4.1 from \cite{BK94}.

\begin{theorem} \label{Bloom Kerman P2}
Consider $K(x,y) \in M_+ \left( \mathbb{R_+}\times \mathbb{R_+} \right)$, which, for fixed $y \in \mathbb{R_+}$, increases in $x$ and, for fixed $x \in \mathbb{R_+},$ decreases in $y$ and which, moreover, satisfies the growth condition 
\begin{equation*}
    K(x,y) \leq K(x,z)+K(z,y), \ \ \  0<y<z<x.
\end{equation*}
Let $t,u,v$ and $w$ be nonnegative, measurable (weight) functions on $\mathbb{R_+}$ and suppose $\Phi_1(x)= \int_0^x \phi_1$  and $\Phi_2(x)= \int_0^x \phi_2$  are $N$-functions having complementary functions $\Psi_1(x)= \int_0^x {\phi_1}^{-1} $ and $\Psi_2(x)= \int_0^x {\phi_2}^{-1}$, respectively, with  $\Phi_1\circ {\Phi_2}^{-1}$ convex. Then there exists 
$c>0$ such that
\begin{equation*}
{\Phi_1}^{-1} \left(   \int_{\mathbb{R_+}}   \Phi_1   \left(   c w(x) \int_0^x K(x,y) f(y) dy   \right) t(x) dx \right) \leq {\Phi_2}^{-1} \left( \int_{\mathbb{R_+}} \Phi_2 \left( \, u(y) f(y) \,  \right)  v(y) dy \right),
\end{equation*}
$f \in M_+({\mathbb{R_+}})$, if and only if 
\begin{equation*}
\int_0^x \frac{K(x,y)}{u(y)} {\phi_2}^{-1} \left( \frac{c \alpha(\lambda,x) K(x,y)}{\lambda u(y)v(y)} \right)dy \leq c^{-1}\lambda 
\end{equation*}
and
\begin{equation*}
\int_0^x \frac{1}{u(y)}\phi^{-1}_2 \left( \frac{ c \beta(\lambda,x)}{\lambda u(y)v(y)} \right) {dy}\leq c^{-1} \lambda,
\end{equation*}
where
\begin{equation*}
\begin{split}
\alpha(\lambda,x) & = \Phi_2\circ {\Phi_1}^{-1}\left(\int_x^{\infty} \Phi_1 \left(  \lambda w(y)   \right)  t(y)dy \right)
\end{split}
\end{equation*}
and 
\begin{equation*}
\begin{split}
\beta(\lambda,x) & = \Phi_2\circ {\Phi_1}^{-1}\left(\int_x^{\infty}\Phi_1  \left( \lambda w(y)K(y,x) \right) t(y)dy \right).
\end{split}
\end{equation*}
\end{theorem}

\begin{theorem}
Let $p,q,u,u',u_p,v,v'$ be as in the \Cref{Gamma bw L2 and L infinity}, with $1<q \leq p<\infty$. 
Then , given that 
 $\Gamma_{p,u}(\mathbb{R_+})$ is an interpolation space between $L_{2}(\mathbb{R_+})$ and $L_{\infty}(\mathbb{R_+})$, one has (\ref{F on Lorentz-p-q}) if and only if 

\begin{enumerate}
\item \label{Con1}
\begin{equation*}
    {\left(\int_0^x u_p(y)dy\right)}^{\frac{1}{p}} 
    {\left(\int_x^{\infty}v'(y)y^{-q'} dy \right)}^{\frac{1}{q'}}
    \leq C 
\end{equation*}
  \item \label{Con2}
\begin{equation*}
     {\left(\int_0^x v'(y) dy \right)}^{\frac{1}{q'}} {\left(\int_x^{\infty}u_p(y)y^{-p} dy \right)}^{\frac{1}{p}}\leq C 
\end{equation*}
   \item \label{Con3}
\begin{equation*}
    {\left(\int_0^x {\left(\log \frac{x}{y} \right) }^p u_p(y) dy \right)}^{\frac{1}{p}}
    {\left(\int_x^{\infty}v'(y)y^{-q'} dy \right)}^{\frac{1}{q'}}
    \leq C
\end{equation*}
    \item \label{Con4}
\begin{equation*}
    {\left(\int_0^x u_p(y)dy\right)}^{\frac{1}{p}} {\left ( \int_x^{\infty} v'(y){\left(    \frac{1}{y} \log \frac{y}{x} \right)}^{q'} dy \right)}^{\frac{1}{q'}}\leq C
\end{equation*}
\end{enumerate}
 Indeed (\ref{Con1}) and (\ref{Con3}) can be combined into
 \begin{equation*}
 {\left(\int_0^\infty { \left(\log \left( 1+\frac{x}{y} \right)  \right) }^p u_p(y) dy \right)}^{\frac{1}{p}}
    {\left(\int_x^{\infty}v'(y)y^{-q'} dy \right)}^{\frac{1}{q'}}
    \leq C
    \end{equation*}
\end{theorem}

\begin{proof} The first inequality in (\ref{on nonnegative from P con cone}) amounts to 
\begin{equation*}
\left[\int_{\mathbb{R_+}}(P g )^{q'} v'\right]^{\frac{1}{q'}} \leq C 
\left[\int_{\mathbb{R_+}}  g^{p'} u_{p}^{1-p'}\right]^{\frac{1}{p'}}
\end{equation*}
and
\begin{equation*}
\left[\int_{\mathbb{R_+}}(Q g )^{q'} v'\right]^{\frac{1}{q'}} \leq C 
\left[\int_{\mathbb{R_+}}  g^{p'} u_{p}^{1-p'}\right]^{\frac{1}{p'}}, g \in M_+(\mathbb{R_+}),
\end{equation*}
the latter inequality being, by duality, equivalent to
\begin{equation*}
\left[\int_{\mathbb{R_+}}(Pf)^{p} \, u_p\right]^{\frac{1}{p}}\leq C 
\left[\int_{\mathbb{R_+}}{f}^{q} \, {v'}^{1-q}\right]^{\frac{1}{q}}, f \in M_+(\mathbb{R_+}).
\end{equation*}

We illustrate the method of proof with the second inequality in (6.7) involving
\begin{equation*}
(P^2g)(x)= \frac{1}{x} \int_0^x \log \left( \frac{x}{y}  \right) g(y) dy.
\end{equation*}
Thus, taking, in Theorem 6.3,
$K(x,y)={\log}_+ \frac{x}{y}$, $\Phi_1(x)=x^{q'}$, $\Phi_2(x)=x^{p'}$
(observe that $(\Phi_1 \circ {\Phi_2}^{-1})(x)= x^{\frac{q'}{p'}},$ which is convex when $q \leq p$), $w(y)=y^{-1}$, $t(y)=v'(y)$, $ u(y)={u_p(y)}^{-1}$, $v(y)=u_p(y)$ we get 
\begin{equation*}
\alpha(\lambda,x) = \lambda^{p'} {\left(\int_{x}^\infty v'(y) y^{-q'} dy \right)}^{\frac{p'}{q'}}
\end{equation*}
and
\begin{equation*}
\beta(\lambda,x) = \lambda^{p'} {\left( \int_{x}^\infty v'(y) \left( y^{-1}\log  \frac{y}{x} \right)^{q'} dy \right)}^{\frac{p'}{q'}}
\end{equation*}
 from which the conditions in \Cref{Bloom Kerman P2} yields (\ref{Con3}) and (\ref{Con4}). We point out that $\lambda$ cancels.

\end{proof}

The inequality (\ref{F on Lorentz-p-q}) is much easier to deal with when
\begin{equation} \label{Gamma Lambda norm equivalence}
    \rho_{p,u} (f) \simeq \lambda_{p,u} (f) = \left(  \int_{\mathbb{R_+}} f^{*}(t)^{p} \, u(t) \, dt  \right)^{\frac{1}{p}},
\end{equation}
which equivalence is not all that uncommon, as we will see later in this section.  Indeed, given (\ref{Gamma Lambda norm equivalence}),
\begin{equation*}
    \begin{split}
        \rho_{p,u} (U f^{*}) & = \left(  \int_{\mathbb{R_+}}  \left( U f^{*}  \right)(t)^{p} \, u(t) \, dt  \right)^{\frac{1}{p}} \\
        & = \left(  \int_{\mathbb{R_+}}  f^{**} (t)^{p} \, u_{p}(t) \, dt  \right)^{\frac{1}{p}}\\
        & = \rho_{p,u_{p}} ( f ).
    \end{split}
\end{equation*}
We thus have

\begin{theorem}
Let $p$, $q$, $u$, $u_{p}$ and $v$ be as in \Cref{Gamma bw L2 and L infinity}. Then, given that $\Gamma_{p,u} (\mathbb{R_+})$ is an interpolation space between $L_{2} (\mathbb{R_+})$ and $L_{\infty} (\mathbb{R_+})$, with $\rho_{p,u}$ satisfying (\ref{Gamma Lambda norm equivalence}), one has
\begin{equation} \label{Optimal inequality F bw Gamma space}
    \rho_{ p, u_{p}}   ( \hat{f}  ) \leq C  \rho_{  p, v}   (f).
\end{equation}
Moreover, there is no essentially larger r.i.-norm that can replace $\rho_{p, u_{p}}$ in (\ref{Optimal inequality F bw Gamma space}).
\end{theorem}

Finally, there is a relatively simple condition sufficient to guarantee (\ref{F on Lorentz-p-q}). It comes out of working with the inequality (\ref{Operator Q L-Gamma-norm ineq}) and involves the norm of the dilation operator $E_s$ as a mapping from $\Gamma_{q,v}(\mathbb{R_+})$ to $\Gamma_{p,u_p}(\mathbb{R_+}),$ namely, 

\begin{equation*}
 h(\Gamma_{q,v},\Gamma_{p,u_p})(t)= \inf \left\lbrace M>0:
\bar{ \rho}_{p, u_{p}}(f(ts))= \bar{\rho}_{p, u_{p}} \left( (E_tf)(s) \right) \leq M \bar{\rho}_{{q, v}}(f) <  \infty \right\rbrace.
\end{equation*}

The argument  in the proof of Theorem 4.1 of \cite{KS21} ensures  (\ref{Operator Q L-Gamma-norm ineq}) provided

\begin{equation*}
\int_1^ \infty h(\Gamma_{q,v},\Gamma_{p,u_p})(t)\,  \frac{dt}{t} < \infty.
\end{equation*}

Again the argument in the proof of Theorem 5.2 in \cite{GK14} yields
\begin{equation*}
h(\Gamma_{q,v},\Gamma_{p,u_p})(t)= \sup_{s>0} \frac{    \left[ \int_0^{s/t} u_p + (s/t)^p\int_{s/t}^\infty u_p(y)y^{-p} dy \right]^{\frac{1}{p}} }  {{\left[\int_0^s v+s^q \int_s^\infty v(y)y^{-q} dy\right]}^{\frac{1}{q}}},
\end{equation*}
when $1<q\leq p<\infty.$

Altogether, we have
\begin{theorem} \label{sufficiency for F on LG spaces}
Let $p$, $q$, $u$, $u_{p}$ and $v$ be as in \Cref{Gamma bw L2 and L infinity}, with $1<q\leq p<\infty$. Then, given that $\Gamma_{p,u} (\mathbb{R_+})$ is an interpolation space between $L_{2} (\mathbb{R_+})$ and $L_{\infty} (\mathbb{R_+})$, one has (\ref{Optimal inequality F bw Gamma space}) 
provided 
\begin{equation*}
\int_{1}^{\infty}  \sup_{s>0} \frac{    \left[ \int_0^{s/t} u_p + (s/t)^p\int_{s/t}^\infty u_p(y)y^{-p} dy \right]^{\frac{1}{p}} }  {{\left[\int_0^s v+s^q \int_s^\infty v(y)y^{-q} dy\right]}^{\frac{1}{q}}} \frac{dt}{t} < \infty.
\end{equation*}
\end{theorem}

\begin{proof} The result follows from the preceeding discussion, since (\ref{Optimal inequality F bw Gamma space}) and (\ref{Operator Q L-Gamma-norm ineq}) are equivalent when $\Gamma_{p,u} (\mathbb{R_+})$ is an interpolation space between $L_{2} (\mathbb{R_+})$ and $L_{\infty} (\mathbb{R_+}).$
    
\end{proof}
 We now consider the question of when 
 $X=\Gamma_{p,u} (\mathbb{R_+})$
 is an interpolation space between $L_{2} (\mathbb{R_+})$ and $L_{\infty} (\mathbb{R_+}).$
    
For $p \in [2,\infty)$ a simple necessary and sufficient condition is given in

\begin{theorem} \label{characterization Gamma space bw L2 Linfinity}
Fix $p \in [2, \infty)$ and suppose $u \in M_+(\mathbb{R_+})$ satisfies $\int_{\mathbb{R_+}} \frac{u(t)}{(1+t)^{p/2}} dt < \infty$. Then, $\Gamma_{p,u}(\mathbb{R_+})$ and $L_{   \rho_{p/2, u}^{(2)}   } (\mathbb{R_+})$ are equal as sets ( and hence the former is an interpolation space between $L_{2}(\mathbb{R_+})$ and $L_{\infty}(\mathbb{R_+})$  ) if and only if
\begin{equation} \label{condition for btwn L2 and L infity}
    t^{p/2} \int_{t}^{\infty} u(s) \frac{ds}{s^{p/2}} \leq C \int_{0}^{t} u, \ \ \ t \in \mathbb{R_+}.
\end{equation}
\end{theorem}
\begin{proof}
 The condition (\ref{condition for btwn L2 and L infity}) is necessary and sufficient in order that
 \begin{equation*}
     \rho_{p/2, u} (f^{*}) \simeq \lambda_{p/2, u} (f^{*}) = \left[ \int_{\mathbb{R_+}}  f^{*}(t)^{\frac{p}{2}} u(t) dt \right]^{\frac{2}{p}}, \ \ \ f \in M_+(\mathbb{R_+}).
 \end{equation*}
 Replacing $f^{*}(t)$ by $f^{*}(t)^{2}$, this equivalence yields
 \begin{equation*}
     \begin{split}
         \rho_{p/2, u}^{(2)} (f^{*}) & = \left[  \int_{\mathbb{R_+}} \left( t^{-1} \int_{0}^{t} f^{*}(s)^{2} ds    \right)^{ \frac{p}{2}} \, u(t) \, dt  \right]^{\frac{2}{p} \cdot \frac{1}{2}}\\
        &  \simeq \left[ \int_{\mathbb{R_+}}  ( f^{*}(t)^2 )^{\frac{p}{2}} \, u(t) \, dt \right]^{\frac{1}{p}}\\
        &  = \left[ \int_{\mathbb{R_+}}   f^{*}(t)^p \, u(t)\,  dt \right]^{\frac{1}{p}} \\
        & =  \lambda_{p,u} (f^{*}).
     \end{split}
 \end{equation*}
 But, (\ref{condition for btwn L2 and L infity}) implies
 \begin{equation*}
     t^p \int_{t}^{\infty} u(s) \frac{ds}{s^p} = \int_{t}^{\infty} u(s)   \left( \frac{t}{s} \right)^{p} ds \leq \int_{t}^{\infty} u(s) \left( \frac{t}{s} \right)^{\frac{p}{2}} ds \leq C \int_{0}^{t} u, \ \ \ t \in \mathbb{R_+},
 \end{equation*}
 and so
 \begin{equation*}
     \lambda_{p,u} (f^{*}) \simeq \rho_{p,u} (f^{*}), \ \ \ f \in M_+(\mathbb{R_+}).
 \end{equation*}
 We conclude $\Gamma_{p,u}(\mathbb{R_+}) = L_{   \rho_{p/2, u}^{(2)}   } (\mathbb{R_+})$ as sets.
\end{proof}

In the proof of Theorem 6.9 below  we require a corollary of the following result of R. Sharpley from  \cite[Lemma 3.1, Corollary 3.2]{S72}

\begin{theorem} \label{Lambdap spaces norm equivalence}
Let $\rho$ be an r.i. norm on $M_+(\mathbb{R}^{n})$. Suppose the fundamental indices of $L_{ \bar{\rho}} (\mathbb{R_+})$ lie in $(0,1)$. Given $p \in (1, \infty)$ set $u_{p}(t) = \frac{  \bar{\rho} \left(  \chi_{(0,t)}\right)^{p}       }{t}$, $t \in \mathbb{R_+}$. Then, $\bar{\rho}_{p, u_{p}} \left(  \chi_{(0,t)}\right)= \bar{\rho} \left(  \chi_{(0,t)}\right)$, $t \in \mathbb{R_+}$. Moreover,
\begin{equation*}
    \bar{\rho}_{p, u_{p}}(f^{*}) \simeq \lambda_{p, u_{p}}(f^{*}), \ \ \ f \in M_+(\mathbb{R_+}).
\end{equation*}
\end{theorem}

\begin{corollary} \label{Lambda=Gamma}
Let $\rho = \rho_{p, u}$ be as in \Cref{Lambdap spaces norm equivalence}. Then, $\rho = \rho_{p, u_{p}}$, where $u_{p}(t) = \frac{  \bar{\rho} \left(  \chi_{(0,t)}\right)^{p}       }{t},$ $t \in \mathbb{R_+}$. 
\end{corollary}
\begin{proof}
The spaces $\Gamma_{p, u}(\mathbb{R}^n)$ and $\Gamma_{p, u_{p}}(\mathbb{R}^n)$ have $\bar{\rho}_{p, u_{p}} \left(  \chi_{(0,t)}\right)= \bar{\rho}_{p,u} \left(  \chi_{(0,t)}\right)$, $t \in \mathbb{R_+}$. As such, the spaces are identical, in view of \cite[Theorem 5.1]{GK14}.
\end{proof}

\begin{theorem} \label{Gamma space bw L2 Linfinity2}
    Fix  $p \in [2, \infty)$ and $u \in M_+(\mathbb{R_+})$, with $\int_{R_+} \frac{u(t)}{(1+t)^{p}}dt < \infty$. Suppose the fundamental indices of $\Gamma_{p,u}(\mathbb{R}^{n})$ lie in $(0,1)$. Then, $\Gamma_{p,u}(\mathbb{R}^n)$ is an interpolation space between $L_{p}(\mathbb{R}^n)$ and $L_{\infty}(\mathbb{R}^n)$ (and hence between $L_{2}(\mathbb{R}^n)$ and $L_{\infty}(\mathbb{R}^n)$) if and only if 
    \begin{equation} \label{characterization}
        \sup_{s \geq t} \frac{\bar{\rho}_{p,u} \left(  \chi_{(0,s)}\right)^{p}}{s} \leq C \,  \frac{\bar{\rho}_{p,u} \left(  \chi_{(0,t)}\right)^{p}}{t},
    \end{equation}
    for some $C>0$ independent of $t \in \mathbb{R_+}$.
\end{theorem}

\begin{proof}
Suppose first that $p=2$. Given $ T : L_{2}(\mathbb{R}^n), L_{\infty}(\mathbb{R}^n) \rightarrow L_{2}(\mathbb{R}^n), L_{\infty}(\mathbb{R}^n)$ one has, according to  \cite[ Theorem 1.11, p. 301]{BS88} and \cite{H70},
\begin{equation*} \label{generalized HLP 2}
     \int_{0}^{t} (T f)^{*} (s)^{2} ds \leq C' M_{2}^{2}  \int_{0}^{Mt} f^{*} (s)^{2} ds 
    = C' M_{2}M_{\infty}   \int_{0}^{t} f^{*} (Ms)^{2} ds,
\end{equation*}
$f \in \left(L_{2} + L_{\infty} \right) (\mathbb{R_+})$, in which $M= M_{\infty} / M_{2}$, $M_{k}$ being the norm of $T$ on $L_{k}(\mathbb{R_+})$, $k=2, \infty$.
In view of (\ref{characterization}), HLP yields
\begin{equation*}
     \begin{split}
         \int_{\mathbb{R_+}}  (Tf)^{*}(t)^{2} \,  \frac{\bar{\rho}_{2,u} \left(  \chi_{(0,t)}\right)^{2}}{t} dt
          & \leq  \int_{\mathbb{R_+}}  (Tf)^{*}(t)^{2} \sup_{s \geq t}  \frac{\bar{\rho}_{2,u} \left(  \chi_{(0,s)}\right)^{2}}{s} dt\\
         &  \leq C' M_{2} M_{\infty} \int_{\mathbb{R_+}}  f^{*}(Mt)^{2} \sup_{s \geq t}  \frac{\bar{\rho}_{2,u} \left(  \chi_{(0,s)}\right)^{2}}{s} dt\\
         & \leq C C' M_{2} M_{\infty} \int_{\mathbb{R_+}}  f^{*}(Mt)^{2}  \, \frac{\bar{\rho}_{2,u} \left(  \chi_{(0,t)}\right)^{2}}{t} dt,
     \end{split}
\end{equation*}
$f \in \left(L_{2} + L_{\infty} \right) (\mathbb{R_+})$. \Cref{Lambdap spaces norm equivalence} now ensures the latter is equivalent to
\begin{equation*}
      \int_{\mathbb{R_+}}  (Tf)^{**}(t)^{2} \,  \frac{\bar{\rho}_{2,u} \left(  \chi_{(0,t)}\right)^{2}}{t} dt \leq C C' M_{2} M_{\infty} h(M)^2 \int_{\mathbb{R_+}}  f^{**}(t)^{2}  \,  \frac{\bar{\rho}_{2,u} \left(  \chi_{(0,t)}\right)^{2}}{t} dt,
\end{equation*}
where $h(t)$ is the norm of the dilation operator  $E_{t}$ on $\Gamma_{ 2, u_{2}   } (\mathbb{R_+}) = \Gamma_{2,u}(\mathbb{R_+})$, $u_{2}(s)= \frac{\bar{\rho}_{2,u} \left(  \chi_{(0,s)}\right)^{2}}{s}$, by 
 \Cref{Lambda=Gamma}, that is, $T : \Gamma_{2,u}(\mathbb{R}^n) \rightarrow  \Gamma_{2,u}(\mathbb{R}^n)$. Thus, $\Gamma_{2,u}(\mathbb{R}^n)$ is between $L_{2}(\mathbb{R}^n)$ and $L_{ \infty}(\mathbb{R}^n)$.
 
 Suppose, next, $p>2$. The ``if" part of our theorem will follow in this case if we can show (\ref{characterization})  implies (\ref{condition for btwn L2 and L infity}), with $u(t)= \frac{\bar{\rho}_{p,u} \left(  \chi_{(0,t)}\right)^{p}}{t}$. But,
 \begin{equation*}
   \begin{split}
         t^{p/2} \int_{t}^{\infty} \frac{\bar{\rho}_{p,u} \left(  \chi_{(0,s)}\right)^{p}}{s} \frac{ds}{s^{p/2}} 
     & \leq t^{p/2} \int_{t}^{\infty}  \sup_{y \geq s} \frac{\bar{\rho}_{p,u} \left(  \chi_{(0,y)}\right)^{p}}{y} \frac{ds}{s^{p/2}}\\
     & \leq t^{p/2}   \sup_{y \geq t} \frac{\bar{\rho}_{p,u} \left(  \chi_{(0,y)}\right)^{p}}{y} \int_{t}^{\infty}\frac{ds}{s^{p/2}}\\
     & \leq C \, t \, \frac{\bar{\rho}_{p,u} \left(  \chi_{(0,t)}\right)^{p}}{t} \\
     & \leq C^2  \int_{0}^{t} \frac{\bar{\rho}_{p,u} \left(  \chi_{(0,s)}\right)^{p}}{s} ds, \ \ \ t \in \mathbb{R_+}. 
     \end{split}
 \end{equation*}
 This completes the proof of ``if" part.

As for the ``only if" part we rely on a result of  L. Maligranda  \cite{M85} asserting that if $L_{\rho}(\mathbb{R}^n)$ is an interpolation space between $L_{p}(\mathbb{R}^n)$ and $L_{\infty}(\mathbb{R}^n)$, then
\begin{equation}\label{LM-necessity- for bw L2 -Linfinity} 
   \frac{ \bar{\rho} (\chi_{(0,s)}) }{ \bar{\rho} (\chi_{(0,t)}) } \leq C \max \left[ \left( \frac{s}{t}  \right)^{\frac{1}{p}} , 1  \right].
\end{equation}
Indeed, for $t \leq s$, (\ref{LM-necessity- for bw L2 -Linfinity}) yields
\begin{equation*}
    \frac{ \bar{\rho} (\chi_{(0,s)}) }{ \bar{\rho} (\chi_{(0,t)}) } \leq C \left( \frac{s}{t}  \right)^{\frac{1}{p}}
\end{equation*}
or
\begin{equation*}
   \frac{ \bar{\rho} (\chi_{(0,s)})^{p} }{ s } \leq C  \, \frac{ \bar{\rho} (\chi_{(0,t)})^{p}  }{t},  
\end{equation*}
from which (\ref{characterization}) follows.
\end{proof}
To this point the Lorentz Gamma range norms have been equivalent to functionals of the form
\begin{equation*}
    \lambda_{p,u} (f) = \left[  \int_{\mathbb{R_+}} f^{*}(s)^{p} u(t) dt  \right]^{\frac{1}{p}}.
\end{equation*}
This need not be the case for the $\rho_{ 2p, u}  $ in \Cref{F not lambda for Fourier} below.

\begin{lemma} \label{Largest weight for P on decreasing}
Fix $p \in (1, \infty)$ and $u \in M_+(\mathbb{R_+})$, with 
\begin{equation*}
    \int_{\mathbb{R_+}} \frac{ u(t) }{ (1+t)^{p} } dt < \infty.
\end{equation*}
Then,
\begin{equation} \label{Bp condition}
    \left(  \int_{\mathbb{R_+}} f^{**}(t)^{p} \, u(t) \, dt  \right)^{\frac{1}{p}}  \leq \left(  \int_{\mathbb{R_+}} f^{*}(t)^{p} \, u^{(p)}(t) \, dt  \right)^{\frac{1}{p}}, \ \ \ f \in M_+(\mathbb{R_+}),
\end{equation}
where
\begin{equation*}
    u^{(p)}(t) = p t^{p-1} \int_{t}^{\infty} u(s) s^{-p} ds, \ \ \ t \in \mathbb{R_+};
\end{equation*}
 moreover, $u^{(p)}$ is essentially the smallest weight for which  (\ref{Bp condition}) holds.
\end{lemma}
\begin{proof}
It is shown in \cite{N91} that
\begin{equation*}
    \left(  \int_{\mathbb{R_+}} f^{**}(t)^{p} \, u(t) \, dt  \right)^{\frac{1}{p}}  \leq \left(  \int_{\mathbb{R_+}} f^{*}(t)^{p} \, v (t) \, dt  \right)^{\frac{1}{p}}, \ \ \ f \in M_+(\mathbb{R_+}),
\end{equation*}
if and only if
\begin{equation*}
    \int_{0}^{t} u(s) ds + t^{p} \int_{t}^{\infty} u(s) s^{-p} ds \leq C \int_{0}^{t} v, \ \ \ t \in \mathbb{R_+}.
\end{equation*}
But,
\begin{equation*}
    \begin{split}
        \int_{0}^{t} u^{(p)} (s) ds & = \int_{0}^{t} p \, s^{p-1} \int_{s}^{\infty} u(y) y^{-p} dy \, ds\\
        & = \int_{0}^{t} p \, s^{p-1} \int_{s}^{t} u(y) y^{-p} dy \, ds + \left[ \int_{0}^{t} p \, s^{p-1} ds \right] \left[ \int_{t}^{\infty} u(s) s^{-p} ds \right]  \\ 
        & =  \int_{0}^{t} \left( \int_{0}^{y} p \, s^{p-1} ds  \right)    u(y) y^{-p} dy  + t^{p}  \int_{t}^{\infty} u(s) s^{-p} ds   \\ 
        & =  \int_{0}^{t}    u  + t^{p}  \int_{t}^{\infty} u(s) s^{-p} ds.
    \end{split}
\end{equation*}
We conclude that
\begin{equation*}
    \left(  \int_{\mathbb{R_+}} f^{**}(t)^{p} \, u(t) \, dt  \right)^{\frac{1}{p}}  \leq \left(  \int_{\mathbb{R_+}} f^{*}(t)^{p} \, u^{(p)}(t) \, dt  \right)^{\frac{1}{p}}, \ \ \ f \in M_+(\mathbb{R_+}).
\end{equation*}
\end{proof}

\begin{theorem} \label{F not lambda for Fourier}
Let $p$ and $u$ be as in \Cref{Largest weight for P on decreasing}. Then,
\begin{equation*}
    \rho_{  2p, u }   ( \widehat{f} \,  )  =  \bar{ \rho}_{  2p, u }   ( ( \widehat{f} )^{*}  )   \leq C  \bar{\rho}_{  2p, u^{(p)}_{2p} }   ( f^{*} ) = \rho_{  2p, u^{(p)}_{2p} }   ( f ),
\end{equation*}
where
\begin{equation*}
    u^{(p)}_{2p} (t)= u^{(p)} ( t^{-1} )  \, t^{2p -2} = p(t^{-1})^{p-1} \int_{t^{-1}}^{\infty} u(s) s^{-p} ds  \, t^{2p -2} = p t^{p-1} \int_{t^{-1}}^{\infty} u(s) s^{-p} ds.
\end{equation*}
\end{theorem}

\begin{proof}
Applying the  construction in (\ref{Rho-p}) to the functionals in  (\ref{Bp condition})  yields
\begin{equation*}
    \left(  \int_{\mathbb{R_+}} \left( t^{-1} \int_{0}^{t} f^{*}(s)^{2} ds    \right)^{p} \, u(t) \, dt  \right)^{\frac{1}{2p}}  \leq \left(  \int_{\mathbb{R_+}} f^{*}(t)^{2p} \, u^{(p)}(t) \, dt  \right)^{\frac{1}{2p}} = \lambda_{2p, u^{(p)}} (f)
\end{equation*}
Again,
\begin{equation*}
\left( t^{-1} \int_{0}^{t} f^{*}(s) ds    \right)^{2p} \leq    \left( t^{-1} \int_{0}^{t} f^{*}(s)^{2} ds    \right)^{p}
\end{equation*}
by H\"older's inequality.

Hence,
\begin{equation*}
    \begin{split}
        \rho_{  2p, u   } ( \widehat{f} \,  ) 
        &  \leq   \rho_{ p, u   } ( \widehat{f}^{ \   2} \,  )^{1/2} \\
        & \leq C  \rho_{ p, u   } ( (U f^{*})^{ 2} \,  )^{1/2} \\
        & \leq C  \lambda_{  2p, u^{(p)}   } ( U f^{*} ) \\
        & = C \left(  \int_{\mathbb{R_+}} (U f^{*})(t)^{2p} \, u^{(p)}(t) \, dt  \right)^{\frac{1}{2p}} \\
        & = C \left(  \int_{\mathbb{R_+}}  f^{**}(t)^{2p} \, u^{(p)}_{2p}(t) \, dt  \right)^{\frac{1}{2p}} \\
        & = C  \rho_{   2p, u^{(p)}_{2p}   } ( f ).
    \end{split}
\end{equation*}
\end{proof}

\begin{example}
Fix $p$, $1<p< \infty$, and set
\begin{equation*}
    u(t)= \begin{cases}
                 t^{2p-1} \left( \log \textstyle\frac{1}{t}   \right)^{- \alpha}, & 0<t<1,\\
                 t^{p-1- \alpha},  & t>1,
                \end{cases}
\end{equation*}
for some $0<\alpha<1$. Then, $p$ and $u$ satisfy the conditions of \Cref{Largest weight for P on decreasing}. Moreover,
\begin{equation*}
    \rho_{ 2p, u  } (f) \not\simeq \lambda_{2p, u} (f), \ \ \ f \in M_+(\mathbb{R_+}),
\end{equation*}
or, equivalently,
\begin{equation} \label{Gamma not equivalent Lambda}
    t^{2p} \int_{t}^{\infty} u(s) s^{-2p} ds \leq C \int_{0}^{t} u, \ \ \ t \in \mathbb{R_+},
\end{equation}
does not hold.
Indeed, the left hand side of (\ref{Gamma not equivalent Lambda}) is equal to $C t^{2p} \left( \log \textstyle\frac{1}{t}  \right)^{- \alpha + 1}$, while the right hand side is 
\begin{equation*}
    \int_{0}^{t} u =  \int_{0}^{t}  s^{2p-1} \left( \log \textstyle\frac{1}{s}   \right)^{- \alpha} \simeq t^{2p} \left( \log \textstyle\frac{1}{t}   \right)^{- \alpha}, \ \ \ 0<t<1,
\end{equation*}
in view of L'H\^ospital rule. The ratio of the left side to the right side in (\ref{Gamma not equivalent Lambda}) is, essentially, $\log \textstyle\frac{1}{t}$ which $\rightarrow \infty$ as $t \rightarrow 0^{+}$.

\end{example}

\section{Other work}\label{Other works}
 Inequalities involving Fourier transform other than those considered in this paper are weighted Lebesgue inequalities
\begin{equation*} \label{Weighted Lp}
    \left(  \int_{\mathbb{R}^n} \lvert \hat{f}(x) \, u(x) \rvert^{q} \, dx  \right)^{\frac{1}{q}}  \leq C \left(  \int_{\mathbb{R}^n} \lvert f(x) \, v(x) \rvert^{p} \, dt  \right)^{\frac{1}{p}}
\end{equation*}
and weighted Lorentz inequalities
\begin{equation*} \label{Weighted Lorentz}
    \left(  \int_{\mathbb{R_+}}  (\hat{f} \, )^{*}(t)^{q} \, u(t)  \, dt  \right)^{\frac{1}{q}}  \leq C  \left(  \int_{\mathbb{R_+}}  \left( \int_{0}^{1/t} f^{*} \, \right)^{p} \, v(t)  \, dt  \right)^{\frac{1}{p}}. 
\end{equation*}

A brief survey of papers on these inequalities, from the pioneering work of Benedetto-Heinig \cite{BH03} through that of G. Sinnamon \cite{S03} and Rastegari-Sinnamon \cite{RS18}, is given in the paper \cite{NT20} of Nursultanov-Tikhonov.


\bibliographystyle{amsalpha}
\bibliography{refes}

\end{document}